\documentclass{amsart}
\usepackage{MnSymbol}
\usepackage{hyperref}
\usepackage{graphicx} 

\usepackage[normalem]{ulem} 

\usepackage{hyperref}
\hypersetup{
    colorlinks,
    citecolor=black,
    filecolor=black,
    linkcolor=black,
    urlcolor=black
}

\usepackage[all]{xy}
\usepackage{a4wide}
\usepackage{comment}

\DeclareFontFamily{OT1}{pzc}{}
\DeclareFontShape{OT1}{pzc}{m}{it}{<-> s * [1.100] pzcmi7t}{}
\DeclareMathAlphabet{\mathpzc}{OT1}{pzc}{m}{it}
\newcommand{\bigast}{\mathop{\scalebox{1.5}{\raisebox{-0.2ex}{$\ast$}}}}

    \usepackage{etoolbox}
    \patchcmd{\section}{\scshape}{\large\bfseries}{}{}
    \makeatletter
    \renewcommand{\@secnumfont}{\bfseries}
    \makeatother

\usepackage{tikz-cd}
\tikzcdset{arrow style=math font}
\tikzcdset{scale cd/.style={every label/.append style={scale=#1},
    cells={nodes={scale=#1}}}}

\numberwithin{equation}{section}
\newtheorem{theorem}{Theorem}[section]
\newtheorem*{theorem*}{Theorem}
\newtheorem{corollary}[theorem]{Corollary}
\newtheorem{lemma}[theorem]{Lemma}
\newtheorem{proposition}[theorem]{Proposition}
\theoremstyle{definition}
\newtheorem{question}{Question}

\newtheorem{remark}[theorem]{Remark}
\newtheorem{example}[theorem]{Example}

\def\DD{\mathpzc{D}}

\def\CC{\mathpzc{C}}

\def\ZZ{\mathbb{Z}}
\def\VV{\mathcal{V}}
\def\QQ{\mathbb{Q}}
\def\Im{\mathrm{Im}}
\def\Ker{\mathrm{Ker}}
\def\Id{\mathrm{Id}}
\def\IM{\mathcal{IM}}
\def\RR{\mathcal{R}}

\title{A note on the non-existence of functors}

\begin{document}

\author{E. Dror Farjoun}
\email{edfarjoun@gmail.com}
\author{S. O. Ivanov} 
\email{ivanov.s.o.1986@bimsa.cn, ivanov.s.o.1986@gmail.com}

\author{A. Krasilnikov}
\email{avkrasilnikov@edu.hse.ru, kras1lnikoff.av@gmail.com}

\author{A. Zaikovskii }
\email{erifiby@gmail.com}

\begin{abstract}
We consider several types of non-existence theorems for functors. For example, there are no nontrivial functors from the category of groups (or the category of pointed sets, or vector spaces) to any small category. Another type of questions that we consider are questions about  nonexistence of subfunctors and quotients of the identity functor on the category of groups (or abelian groups). For example, there is no a natural non-trivial way to define an abelian subgroup of a group, or a perfect quotient group of a group. 
As an auxiliary result we prove that, for any non-trivial subfunctor $F$ of the identity functor on the category of groups, any group can be embedded into a simple group that lies in the essential image of $F.$ The paper concludes with a few questions regarding the non-existence of certain (co-)augmented functors in the $\infty$-category of spaces. 
\end{abstract}

\maketitle

\section{Introduction}

It is natural to expect that many large and ``rich'' categories should  have only trivial functors to any small category. A typical example of a statement proven below is  that any functor from the category of groups to the category of finite groups is constant. We show that  any functor from a strongly connected (i.e. all hom-sets are non-empty) category $\CC$ with small products to a small category $\DD$ is a constant functor. Moreover, we prove a version of this statement that depends on a cardinal. 
\begin{theorem*}
If there is a cardinal $\kappa$ such that $\CC$ is a strongly connected category that has products indexed by sets of cardinality $\leq \kappa,$ and $\DD$ is an essentially small category whose hom-sets have cardinality $\leq \kappa,$ then any functor from $\CC$ to $\DD$ is constant.  
\end{theorem*}
Interesting examples of  consequences of this theorem are given by the following statements. Any functor from the category 
of countable groups to the category of finitely generated groups is constant. For any two cardinals $\kappa$ and $\kappa'$ such that $\kappa\geq \max(2^{\kappa'},\aleph_0),$ any functor from the category of pointed sets of cardinality at most $\kappa$ to the category of pointed sets of cardinality at most $\kappa'$ is constant.

In a slightly different vein, we consider the impossibility of certain subfunctors and quotients of the identity functor of the category of groups, and the category of abelian groups. For example, the abelianization of a group $G\mapsto G/[G,G]$ is a quotient of the identity functor on the category of groups that takes values in the subcategory of abelian groups. 
So, there is a natural quotient of any group which is abelian. Is it possible to find a natural abelian subgroup in a group $F(G)\subseteq G$ in a non-trivial way? 
The center of a group comes to mind, but it is not a functor. We show that this is impossible, there is no a non-trivial subfunctor of the identity functor on the category of groups taking values in the subcategory of abelian groups. 
In order to prove this statement, we study the essential image of a sub-functor of the identity functor. 
The essential image  $\IM(F)\subseteq {\sf Ob}(\DD)$ of a functor $F:\CC\to \DD$ is defined as the class of objects isomorphic to $Fc$ for some $c.$ It is well known that any group can be embedded into a simple group \cite{schupp1976embeddings}.  The following theorem can be interpreted as a strengthening of this statement.

\begin{theorem*}
Let $F$ be a non-trivial subfunctor of the identity functor on the category of groups, and $G$ be a group. Then there exists a simple group $S$ such that $G\subseteq S$ and $S=FS\in \IM(F).$ 
\end{theorem*}
This theorem implies that for any proper variety of groups $\VV$ and any augmented functor on the category of groups $F:{\sf Gr}\to {\sf Gr}$ with augmentation $\varepsilon:F\to {\rm Id},$ the augmentation $\varepsilon$ is trivial. 

A similar question to the question about a natural abelian subgroup of a group is the question about a natural perfect quotient of a group. There is a natural perfect subgroup of any group called perfect core. Is there a natural perfect quotient? We answer this question in the negative. Moreover, we prove the following theorem. 

\begin{theorem*}
Let $Q$ be a non-trivial quotient of the identity functor on the category of groups. Then $\IM(Q)$ contains either a non-trivial finite cyclic group, or all free abelian groups.  
\end{theorem*}

In particular, we obtain that, for any natural quotient $G\twoheadrightarrow QG$ of a group with $QG$ always a free group, $QG$ is trivial. Similarly if $QG$ is always either trivial or a non-abelian group then $QG$ is identically the trivial group. We also prove some similar statements about the category of abelian groups. 

Note that the subcategory of abelian groups is reflective in the category of groups. There are many other examples of reflective subcategories of groups: nilpotent groups of class $n$, metabelian groups, $n$-Burnside groups, uniquely divisible groups, $HR$-local groups of Bousfield  (for more examples see \cite{akhtiamov2021right}). We show that there is no a non-trivial subfunctor of the identity functor in the category of groups having values in these subcategories. 

\begin{theorem*}
Let $\RR$ be a proper reflective subcategory of the category of groups. Then there is no a non-trivial subfunctor of the identity functor on the category of groups having values in $\RR.$
\end{theorem*}

However, the dual statement does not hold (Example \ref{example:coreflective}), and the analogue of this statement for the category of abelian groups does not hold as well (Example \ref{example:reflective_abelian}).

Our work was partly motivated by questions from homotopy theory. It is natural to expect that there should be analogous non-existence propositions  in the $\infty$-category of spaces. Next, we will propose some ``non-existence'' questions from homotopy theory. We denote by $\mathcal{S}^{\geq 1}_*$ the $\infty$-category of pointed connected spaces, and by $\mathcal{S}p$ the $\infty$-category of spectra. The following question is a natural analogue of the result  about the non-existence of a non-trivial natural abelian subgroup in a group.

\begin{question}
For a functor  $F:\mathcal{S}^{\geq 1}_* \to \mathcal{S}p$ from the $\infty$-category of spaces to spectra, is every natural transformation to the identity functor $ \Omega^\infty F X \to X$ null-homotopic?
\end{question}

\begin{question}

 Let $F:\mathcal{S}^{\geq 1}_* \to \mathcal{S}^{\geq 1}_*$ be a augmented  functor from the $\infty$-category of pointed connected spaces to itself and $n\geq 1$ be an integer such that the map $\pi_n(FX)\to \pi_n(X)$ is trivial for all $X$. Does it follow that the augmentation   
$FX\to X$ naturally and uniquely factors homotopically through the $(n+1)$-th stage of the Whitehead tower    
$X^{\geq n+1}\to X.$ Or, more technically, the $n$-connected cover is the terminal  augmented functor, $FX\to X,$ with $\pi_n (FX)\to \pi_n (X)$ is the zero map (note that here the zero map is assumed only for one $n$). In a slogan, ``if $FX\to X$ kills one homotopy group, it also kills all the lower ones''. 
\end{question}

Similar results about coaugmented functors can be found in \cite{farjoun1996higher}. From each of the above implication one deduces immediately that, on the $\infty$-category of spaces, there are only trivial natural transformations $F\to {\rm Id}$ such that and for all $X$ one has $FX\cong K(A,n)$  for some $n\geq 1.$ Note that the last space has the structure of abelian group, thus this implication is a direct analog of Theorem \ref{noabeliansubgroups} below.


\section{Functors from large to small categories}

In preparation for the  propositions below, we establish some notations.
Let $\CC$ be a category. For any family of objects ${\bf c}=(c_\alpha)_{\alpha\in A}$ such that the product $\prod_{\alpha} c_\alpha$ exists, and any $\beta \in A$ we denote by 
\begin{equation}
  p^{\bf c}_\beta : \prod_{\alpha} c_\alpha \longrightarrow c_\beta  
\end{equation}
the canonical projection from the product. For any object $c'$ we denote by 
\begin{equation}
t : \prod_{\alpha} \CC(c',c_\alpha) \longrightarrow \CC\left(c', \prod_{\alpha} c_\alpha \right) 
\end{equation}
the standard bijection, which is uniquely defined by the formula 
\begin{equation}\label{eq:t}
p^{\bf c}_\beta \circ t(\varphi) = \varphi_\beta  
\end{equation}
for any family $\varphi = (\varphi_\alpha:c\to c_\alpha)_{\alpha\in A}$ and any $\beta\in A.$ 

\smallbreak

Let $\CC,\DD$ be two categories,  and let $F:\CC\to \DD$ be a functor. Assume that  ${\bf c}=(c_\alpha)_{\alpha\in A}$ is a family of objects of $\CC$ such that the products $\prod_{\alpha} c_\alpha$ and $\prod_\alpha F(c_\alpha)$ exist. Consider the standard assembly morphism
\begin{equation}
 v^{\bf c}:F\left( \prod_{\alpha} c_\alpha \right) \longrightarrow \prod_{\alpha} F(c_\alpha) 
\end{equation}

such that $p^{F({\bf c})}_\beta \circ  v^{{\bf c}} = F(p^{{\bf c}}_\beta)$ for any $\beta \in A.$ Here we denote by $F({\bf c})$ the family $(F(c_\alpha))_{\alpha\in A}.$
\\

The following is a key observation,
saying that for any  $F:\CC\to \DD$ as above, certain  natural maps between "large products" factors through a mapping set inside $\DD:$
\\
\begin{lemma}\label{lemma:commutative_diagram} For any object $c'$ of $\CC$ and any family ${\bf c}=(c_\alpha)_{\alpha\in A}$ of objects of $\CC$ such that the products $\prod_{\alpha} c_\alpha$ and $\prod_{\alpha} F(c_\alpha)$ exist the diagram 
 \[
\begin{tikzcd}
\prod\limits_{\alpha} \CC(c',c_\alpha) \ar[rr,"\prod F"] \ar[d,"t","\cong"'] 
& 
&  
\prod\limits_{\alpha} \DD(F(c'),F(c_\alpha)) \ar[d,"t","\cong"']
\\
\CC\left(  c' , \prod\limits_{\alpha} c_\alpha \right) \ar[r,"F"] 
&
\DD\left( F( c') , F\left( \prod\limits_{\alpha} c_\alpha \right)\right) 
\ar[r,"{v^{\bf c}_*}"]
& 
\DD( F(c'), \prod\limits_\alpha F(c_\alpha) )
\end{tikzcd}
\]
is commutative.  In other words, for any family of morphisms $\varphi=(\varphi_{\alpha}:c' \to c_\alpha)_{\alpha\in A}$ we have
\[
v^{\bf c} \circ F(t(\varphi))  = t({F(\varphi)}).
\] 
\end{lemma}
\begin{proof}
It follows from the defining property of $t(F(\varphi))$ \eqref{eq:t} and the formula
\begin{equation}
p^{F({\bf c})}_{\beta} \circ v^{\bf c} \circ F(t(\varphi))   = F(p^{\bf c}_\beta\circ t(\varphi))  = F(\varphi_{\beta}).
\end{equation}
\end{proof}

The following non-existence results are  quick implications of the above lemma.  
A category $\CC$ is called \emph{strongly connected} if for any two objects $c,c'\in \CC$ the hom-set $\CC(c,c')$ is not empty.

\begin{proposition}\label{proposition:main1} Let $\kappa$ be a cardinal, 
 $\CC,\DD$ be categories with products indexed by sets of cardinality $\kappa,$ and let $F:\CC\to \DD$ be a functor. Assume that $|\DD(F(c'),F(c))|\leq \kappa$ for any two objects $c',c$ of $\CC.$ Then the following holds
\begin{enumerate}
\item for any two parallel morphisms $f,g:c'\to c$ of $\CC$ we have $F(f)=F(g);$
\item if $\CC$ is strongly connected, then $F$ is isomorphic to a constant functor.
\end{enumerate}
\end{proposition}
\begin{proof}
(1) Assume the contrary that there exists morphisms $f,g:c'\to c$ such that $F(f)\ne F(g).$ Take a set $A$ such that $|A|=\kappa$ and consider a constant family ${\bf c}=(c_\alpha)_{\alpha\in A},$ where $c_\alpha=c.$ Since the image of the map $\CC(c',c_\alpha)\to \DD(F(c'), F(c_\alpha))$ has at least two elements, the cardinality of the image of the map $\prod \CC(c',c_\alpha)\to \prod \DD(F(c'), F(c_\alpha))$ has cardinality at least $2^{\kappa}>\kappa.$ On the other hand, by Lemma \ref{lemma:commutative_diagram} the map $\prod \CC(c',c_\alpha)\to \prod \DD(F(c'),F(c_\alpha))$ factors through the set $\DD(F(c'),F\left(\prod c_\alpha \right) ) $ whose cardinality $\leq \kappa$ by the assumption, which is impossible.

(2) Since $\CC$ is strongly connected, it is sufficient to prove that for any morphism $f:c'\to c$ the morphism $F(f)$ is an isomorphism. Take a morphism $g:c\to c'.$ By (1) we obtain $F(g)F(f) = F(gf)=F({\sf id}_{c'})={\sf id}_{F(c')}$ and $F(f)F(g)=F(fg)=F({\sf id}_c)={\sf id}_{F(c)}.$ Then $F(f)$ is an isomorphism.
\end{proof}

A category $\DD$ is called essentially small, if it is equivalent to a small category.

\begin{theorem}\label{theorem:main2}
Let $\CC,\DD$ be categories and $F:\CC\to \DD$ be a functor. Assume that $\DD$ is essentially small category and there exists a cardinal $\kappa$ such that 
 $\CC$ has products indexed by any set of cardinality $\kappa$ while $|\DD(d',d)|\leq \kappa$ for all objects $d',d$ of $\DD.$ Then the following holds
\begin{enumerate}
\item for any two parallel arrows $f,g:c'\to c$ we have $F(f)=F(g);$
\item if $\CC$ is strongly connected, then $F$ is isomorphic to a constant functor.
\end{enumerate}
\end{theorem}
\begin{proof} Without loss of generality we can assume that $\DD$ is small. 
 Consider the category of presheaves $\hat \DD={\sf Funct}(\DD^{op},{\sf Sets})$ and the Yoneda embedding $Y:\DD\to \hat \DD.$  By the Yoneda lemma we obtain that \[|\hat \DD(Y(d'),Y(d))|=|\DD(d',d)|\leq \kappa\] for any objects $d',d$ of $\DD.$ We also know that $\hat \DD$ has all small products. Then we can apply Proposition \ref{proposition:main1} to the functor $YF:\CC\to \hat \DD.$ The assertion follows from the fact that $Y$ is fully faithful. 
\end{proof}

\begin{corollary}[{cf. \cite[Th.2.1]{shulman2008set}}]
Any small category with small products is a preorder. 
\end{corollary}
\begin{proof}
Take $\CC=\DD,$ $F=\Id$ and $\kappa=|{\sf Mor}(\CC)|,$ and apply Theorem \ref{theorem:main2}(1).
\end{proof}

\begin{corollary}\label{cor: constant} Let $\CC$ be a strongly connected category with small products and $\DD$ be an essentially small category. Then any functor $ \CC\to \DD$ is isomorphic to a constant functor.   
\end{corollary}

\begin{remark} There are dual versions of all of the above statements in this section, where products are replaced by coproducts. 
\end{remark}

\begin{corollary} Let $\CC$ be one of the following categories:
\begin{itemize}
\item the category of non-empty sets;
\item the category of pointed sets;
\item the category of groups;
\item the category of modules over some ring;
\end{itemize}
and let $\DD$ be one of the following categories
\begin{itemize}
\item finite sets;
\item finite groups;
\item countable groups;
\item finitely generated modules over a ring $R.$
\end{itemize}
Then any functor from $\CC$ to $\DD$ is isomorphic to a constant functor.
\end{corollary}

\begin{example} 
Denote by $O$ the category, whose objects are the ordinals and morphisms are canonical embeddings. So $O$ is a ``big poset,'' in fact  a ``big well ordered set.'' Then $O$ has small products defined by the infimum and small coproducts defined by the supremum. However, $O$ is not strongly connected, and for any ordinal $\alpha>0$ there is a non-constant functor to the small sub-category $\alpha+1$ of $O:$
\begin{equation}
    F : O \to \alpha+1, \hspace{1cm} F(\beta):={\sf min}(\beta,\alpha).
\end{equation}
Therefore, the assumption of being strongly  connected in Corollary \ref{cor: constant} is essential. 
\end{example}

\begin{proposition}\label{contabletofggroups}
Let ${\sf Gr}^{\sf count}$ be the category of countable groups and ${\sf Gr}^{\sf f.g.}$ be the category of finitely generated groups. Any functor ${\sf Gr}^{\sf count}\to {\sf Gr}^{\sf f.g.}$ is isomorphic to a constant functor.   
\end{proposition}
\begin{proof}
Countable free product of countable groups is countable. Hence ${\sf Gr}^{\sf count}$ has countable coproducts. The hom-sets between finitely generated groups are countable, because any homomorphism is uniquely defined by its values of generators. Therefore, if we take $\kappa=\aleph_0$ and $\CC={\sf Gr}^{\sf count},$ and $\DD={\sf Gr}^{\sf f.g.}$ and use the dual version of the Proposition \ref{proposition:main1}, we obtain that any functor ${\sf Gr}^{\sf count}\to {\sf Gr}^{\sf f.g.}$ is isomorphic to a constant functor. 
\end{proof}

For a cardinal $\kappa$ we denote by ${\sf Sets}^{\leq \kappa}_{\ne \emptyset}$ the category of non-empty sets of cardinality $\leq \kappa$ and by ${\sf Sets}^{\leq \kappa}_*$ the category of pointed sets of cardinality $\leq \kappa.$ Note that the categories ${\sf Sets}^{\leq \kappa}_{\ne \emptyset}$ and ${\sf Sets}^{\leq \kappa}_*$ are strongly connected, while the category of all sets ${\sf Sets}$ is not strongly connected, because there is no  map from non-empty set to the empty set. 

\begin{proposition}\label{setstosmallerset}
Let $\kappa,\kappa'$ be two  cardinals such that $\kappa\geq 2^{\kappa'}.$ Assume that $\kappa$ is infinite. Then all functors
\[
{\sf Sets}^{\leq \kappa}_{\ne \emptyset}\longrightarrow {\sf Sets}^{\leq \kappa'}_{\ne \emptyset}, \hspace{1cm} {\sf Sets}^{\leq \kappa}_{*}\longrightarrow {\sf Sets}^{\leq \kappa'}_{*}
\] 
are isomorphic to constant functors. 
\end{proposition}
\begin{proof} We prove the statement for the category of non-empty sets, the case of pointed sets is similar. Since $\kappa$ is infinite, we have $\kappa\times \kappa = \kappa,$ and hence, 
the disjoint union of non-empty sets of cardinality $\leq \kappa$ indexed by a set of cardinality $\kappa$ has cardinality $\leq \kappa.$ Hence ${\sf Sets}^{\leq \kappa}_{\ne \emptyset}$ is strongly connected and has coproducts indexed by sets of cardinality $\kappa.$ We claim that for any sets $X,Y$ of cardinality $\leq \kappa'$ the set of functions $X^Y$ has cardinality $\leq \kappa.$ If $\kappa'$ is finite, it is obvious. If $\kappa'$ is infinite, then it is well known that $(\kappa')^{\kappa'}=2^{\kappa'}$ (see  \cite[XV.2]{sierpinski1958cardinal}, \cite[Ch.2, Exercise 4]{kaplansky2020set}) and hence $|X^Y|\leq 2^{\kappa'} \leq \kappa.$ The assertion follows from Proposition \ref{proposition:main1}.
\end{proof}

\begin{example}[Functors from sets]
The category of all sets ${\sf Sets}$ is not strongly connected, because there is no map from a non-empty set to the empty set. So we can't use Proposition  \ref{proposition:main1} and Theorem \ref{theorem:main2} for this category. Indeed, for any  category $\DD$ and any its morphism $\varphi:d'\to d$ there is a functor
\[ F_\varphi : {\sf Sets} \longrightarrow \DD\]
such that 
\[ 
F_\varphi(X) = 
\begin{cases}
d', & X=\emptyset \\
d, & X\ne \emptyset 
\end{cases} , 
\hspace{1cm}
F_\varphi(f:X\to Y) =
\begin{cases}
\varphi, & X=\emptyset,\  Y\ne \emptyset \\ 
1_{d'}, & X=Y=\emptyset \\ 
1_{d}, & X\ne \emptyset,\ Y\ne \emptyset 
\end{cases}
.
\]
If $\DD$ is small, using the fact that all functors ${\sf Sets}_{\ne \emptyset}\to \DD$ are isomorphic to constant functors, it is easy to check that all functors ${\sf Sets}\to \DD$ are isomorphic to functors of the form $F_\varphi.$
\end{example}

Note, however, that not all functors from "larger sets" to "smaller sets" are constant.

\begin{example}
Here for a natural number $n\geq 2$ we give an example of a functor 
\begin{equation}
F : {\sf Sets}^{\leq n}_{\ne \emptyset} \longrightarrow {\sf Sets}^{\leq 2}_{\ne \emptyset}
\end{equation}
which is not isomorphic to a constant functor. 
There are two types of morphisms in ${\sf Sets}^{\leq n}_{\ne \emptyset}$: (0) functions with the image of cardinality $<n;$ (1)  bijections between sets of cardinality $n.$ A composition of a morphism of type (0) with any other morphism from any side is a morphism of type (0). The composition of two morphisms of type (1) is a morphism of type (1). Denote by $f_0:\{0,1\}\to \{0,1\}$ the map sending all elements to zero $f_0(x)=0.$ We define $F$ so that $F(X)=\{0,1\}$ for any set $X;$ for any morphism $f$ of type (0) we set $F(f)=f_0;$ for any morphism $f$ of type (1) we set $F(f)={\sf id}_{\{0,1\}}.$ It is easy to check that $F$ is a well-defined functor non-isomorphic to a constant functor. 
\end{example}

\section{Subfunctors and quotients of the identity functor}

\subsection{Subfunctors of the identity functor of the category of groups} 
 Throughout this subsection, we denote by $F$ a given sub-functor of the identity functor on the category of groups
\begin{equation}
 F\subseteq \Id_{\sf Gr}.   
\end{equation}

A subgroup $H$ of a group $G$ is called \emph{characteristic} (denoted by $H\triangleleft_{\sf ch} G $) if for any automorphism $\varphi: G\to G$ we have $\varphi(H)\subseteq H.$ Any characteristic subgroup is normal because it is invariant with respect to inner automorphisms. An advantage of characteristic subgroups over normal subgroups is that $K\triangleleft_{\sf ch} H $ and 
$ H \triangleleft_{\sf ch} G$ implies $K\triangleleft_{\sf ch} G.$ Note that the injectivity assumption on $F$  namely, $FG\subseteq G,$ implies directly that $F$ preserves injectivity of  maps.

\begin{lemma}\label{lemma:subfunctor} \ 
\begin{enumerate}
\item $F(G) \triangleleft_{\sf ch} G$ for any $G;$
\item for any $H\triangleleft_{\sf ch} F(G)$ we have 
\[
F(G)/H \subseteq F(G/H);
\]
\item for any family of groups $(G_\alpha)_{\alpha\in A}$ the standard assembly map is injective: 
\[\bigast_{\alpha\in A} F(G_\alpha) \subseteq F(\bigast_{\alpha\in A} G_\alpha).\]
\item for any family of groups $(G_\alpha)_{\alpha\in A}$ the standard assembly map is injective: 
\[F(\prod_{\alpha\in A} G_\alpha) \subseteq \prod_{\alpha\in A} F(G_\alpha).\]
\end{enumerate}
\end{lemma}
\begin{proof}
(1) For any automorphism $\varphi$ of $G$ the map $F(\varphi)$ is an automorphism of $F(G).$ Since $F(\varphi)$ is a restriction of $\varphi$ to a subgroup, we obtain $\varphi(F(G))\subseteq F(G).$

(2) Since $H\triangleleft_{\sf ch} F(G)$ and $F(G)\triangleleft_{\sf ch} G,$ we obtain that $H$ is normal in $G.$ Hence we can consider the projection $p:G\to G/H.$ By the naturality of $F$ we obtain that $F(G/H)$ contains $p(F(G))=F(G)/H.$ 

(3) The group $\bigast_{\alpha\in A} F(G_\alpha)$ is the subgroup of $\bigast_{\alpha\in A} G_\alpha$ generated by the images of $F(G_\alpha).$ The assertion follows.

(4) For any projection $p_\beta: \prod_{\alpha\in A} G_\alpha \to G_\beta$ we have $p_\beta(F(\prod_{\alpha\in A} G_\alpha)) \subseteq F( G_\beta ).$ The assertion follows. 
\end{proof}

\begin{corollary}[{cf. \cite[Prop. 3.1, Prop 3.1*]{huq1969semivarieties}}]
\label{cor:classes} \ 
\begin{enumerate}
\item The class of groups $G$ such that $F(G)=G$ is closed with respect free products and taking quotients.
\item The class of groups $G$ such that $F(G)=0$  is closed with respect products and taking subgroups.
\end{enumerate}
\end{corollary}

\begin{corollary}\label{cor:Z}
If $F(\ZZ)=\ZZ,$ then $F=\Id_{\sf Gr}.$
\end{corollary}
\begin{proof}
This follows from Corollary 
\ref{cor:classes}  (1) above. Any free group is a free product of copies of $\ZZ,$ and any group is a quotient of a free group. 
\end{proof}

By the essential image $\IM(F)$ of a functor $F:\CC\to \DD$ we mean the class of objects isomorphic to an object of the form $F(c),c\in {\sf Ob}(\CC).$ 

\begin{theorem}\label{th:simple_image}
Let $F$ be a non-trivial subfunctor of the identity functor on the category of groups, and $G$ be a group. Then there exists a simple group $S$ such that $G\subseteq S$ and $S=FS\in \IM(F).$ 
\end{theorem}
\begin{proof} Assume the contrary, that there is a group $B$ such that for any simple group $S$ containing $B,$ we have $F(S)\neq S.$ Since $F(S)$ is a normal subgroup (Lemma \ref{lemma:subfunctor}), we have $F(S)=1.$ 
Every group can be embedded into a simple group \cite{schupp1976embeddings}. Given a group $G,$  
take the product $G\times B$ and consider a simple group $S_G$ containing $G\times B.$ Then we have  $F(S_G)=1.$   Combining this with the inclusion $F(G)\subseteq F(S_G),$ which is implied by functoriality, we obtain $F(G)=1$ for any $G,$  contradicting the assumption. 
\end{proof}

\begin{remark} This implies, in particular, that for any functor $F$ as above, any group  $G$ can be embedded in a group of the form $FH,$ for some group $H.$
\end{remark}

\smallbreak

Next, let us consider augmented functors $FG\to G,$ taking values in any {\em proper reflective subcategory} of groups. A full subcategory $\DD$ of a category $\CC$ is called reflective, if for any object of $c$ of $\CC$ there is a universal arrow $c\to d$ to an object of $\DD.$ In other words, a full subcategory $\DD\subseteq \CC$ is reflective, if the functor of embedding $\DD \hookrightarrow \CC$ has a left adjoint functor $ L:\CC\to \DD$ called reflector.  We say that a reflective subcategory $\DD$ is  proper, if it is isomorphism-closed and not equal to $\CC.$ 

\begin{example}[Varieties of groups]
A variety of groups is a class of all groups whose elements satisfy a fixed system of identity relations. It is also well known that a variety of groups can be defined as a full subcategory of the category of groups that is closed under small products, taking images and taking subgroups. A variety is called proper if it is not equal to the category of all groups. Examples of proper varieties of groups are abelian groups, nilpotent groups of class $n$, metabelian groups, $n$-Burnside groups (i.e. groups $G$ satisfying $g^n=1$ for all elements $g\in G$). A proper variety of groups is a proper reflective subcategory in the category of groups.      
\end{example}

\begin{example}[$f$-local groups] Let $f:A\to B$ be a homomorphism of groups. We say that a group $G$ is $f$-local, if $f$ induces a bijection $f^*: {\sf Hom}(B,G) \overset{\cong}\to  {\sf Hom}(A,G)$ (see \cite{farjoun1996cellular}, \cite{libman2000cardinality}, \cite{akhtiamov2021right}). A homomorphism $f:A\to B$ is called local, if $B$ is $f$-local. For any local homomorphism $f,$ the full subcategory of $f$-local groups $\RR_f$ is reflective \cite[Cor.1.7]{casacuberta1992orthogonal}. Moreover, the reflector $L_f$ satisfies the property $L_f A\cong B.$ In particular, if $f$ is not an isomorphism, $\RR_f$ is a proper reflective subcategory.  
\end{example}

\begin{example}[Uniquely divisible groups]
A group $G$ is called uniquely divisible, if for any $g\in G$ and any integer $n$ there exists a unique $h\in G$ such that $h^n=g.$ Equivalently, uniquely divisible groups can be defined as groups local with respect to the homomorphism $\ZZ\hookrightarrow \QQ.$ In particular, we obtain that the class of uniquely divisible groups forms a proper reflective subcategory. 
\end{example}

\begin{lemma}\label{lemma:monoreflective}
Let $\RR$ be a proper reflective subcategory of the category of groups with a reflector $L.$ Then the kernel of the unit 
\begin{equation}
KG={\sf Ker}(\eta_G: G \to LG)
\end{equation}
is a non-trivial subfunctor of the identity functor. 
\end{lemma}
\begin{proof} Let us assume that $KG$ is trivial and prove that $\eta_G$ is an isomorphism. Since epimorphisms in the category of groups are surjective homomorphisms, it is sufficient to prove that $\eta_G$ is an epimorphism. Take two homomorphisms $i,i':LG\to A$ such that $i\eta_G = i' \eta_G.$  By the universal property of $LG,$ we obtain that the equations  $\eta_{A} i \eta_G  = \eta_{A} i' \eta_G$ imply $\eta_A i = \eta_A i'.$ Then using that $\eta_A$ is a monomorphism, we obtain $i=i'.$
\end{proof}

\begin{theorem}\label{th:reflective}
Let $\RR$ be a proper reflective subcategory of the category of groups. Then there is no a non-trivial subfunctor of the identity functor on the category of groups having values in $\RR.$
\end{theorem}
\begin{proof}
Assume that $F$ is a non-trivial subfunctor of the identity functor and prove that its values do not lie in $\RR.$ Let $L$ be the reflector for $\RR$ and $KG$ be the kernel of $\eta_G:G\to LG.$ By Lemma \ref{lemma:monoreflective}, $K$ is also a non-trivial subfunctor of the identity functor. Since $F$ and $K$ are non-trivial, there exist groups $G$ and $G'$ such that $FG$ and $KG'$ are nontrivial.  
 Take a simple group $S$ containing the product $G\times G'.$  The inclusion $ G \subseteq S$ induces an inclusion $FG\subseteq FS.$ Therefore $FS$ is non-trivial, and hence, $FS=S.$ Similarly we obtain $KS=S.$ The equation $KS=S$ implies that $S \notin \RR.$ Therefore $FS \notin \RR,$ which finishes the proof.    
\end{proof}

\begin{corollary}\label{noabeliansubgroups} Let $\mathcal V$ be a proper variety of groups (for example, the variety of abelian groups) and $F:{\sf Gr} \to {\sf Gr}$ be an augmented functor 
\begin{equation}
\varepsilon: F \longrightarrow \Id_{\sf Gr} 
\end{equation}
taking values in $\VV.$ Then $\varepsilon$ is trivial. 
\end{corollary}
\begin{proof} Since $\VV$ is a variety, the image $\Im (\varepsilon)\subseteq \Id_{\sf Gr}$ also takes values in $\VV.$ Then the result follows from Theorem \ref{th:reflective} and the fact that $\VV$ is a proper  reflective subcategory. 
\end{proof}

\begin{corollary}
There is no a non-trivial subfunctor of the identity functor on the category of groups having values in uniquely divisible groups. 
\end{corollary}

\subsection{Quotients of the identity functor of the category of groups}

Let us now briefly consider  conditions that any quotient of the identity functor on groups  $G\twoheadrightarrow Q(G)$ must satisfy.

\begin{theorem}\label{th:gr:quotient}
Let $Q$ be a non-trivial quotient of the identity functor on the category of groups. Then $\IM(Q)$ contains either a non-trivial finite cyclic group, or all free abelian groups.
\end{theorem}
\begin{proof} Assume  that $\IM(Q)$ does not contain a non-trivial finite cyclic group and show that all free abelian groups are in  $\IM(Q).$ Set $F(G)=\Ker(G\to Q(G)).$  If $F(\ZZ)=\ZZ,$ then by Corollary \ref{cor:Z}, $Q$ is trivial. 
Assume that $F(\ZZ)\ne \ZZ.$ Using that non-trivial finite cyclic groups are not in $\IM(Q),$ we obtain $F(\ZZ)=0.$ Since $F(\ZZ)=0,$ Corollary \ref{cor:classes} (2) implies that $F(\ZZ^{\oplus X})=0$ for any set $X,$ and $Q(\ZZ^{\oplus X})\cong \ZZ^{\oplus X}.$ 
\end{proof}

\begin{corollary}
There is no a non-trivial quotient of the identity functor on the category of groups taking values in the subcategory of perfect groups (or free groups, or non-abelian groups). 
\end{corollary}

\begin{example}[Quotient of the identity functor having values in a coreflective subcategory] \label{example:coreflective}
Denote by $\mathcal{T}$ the full subcategory of the category of groups generated by elements of finite order. This subcategory is coreflective. Indeed for any group $G,$ we can consider the subgroup $tG\subseteq G$ generated by elements of finite order, and the inclusion $tG \hookrightarrow G$ is the universal arrow from a group generated by torsion elements. On the other hand, we can consider the quotient $QG$ of a group $G$ by the normal subgroup generated by squares of all elements. Therefore $QG$ is a non-trivial quotient of the identity functor having values in the coreflective subcategory $\mathcal{T}.$
\end{example}

\begin{example}[Natural embedding into a perfect group] Although, there is no non-trivial  quotient of the identity functor taking values in the category of perfect groups, there is a natural embedding
\begin{equation}
 G \hookrightarrow C(G),   
\end{equation}
where $C(G)$ is an acyclic (in particular, perfect) group. There are two different constructions of such a functor $C(G):$ one of them uses the group of bijections $\mathbb Q\to \mathbb Q$ with ``compact  support'' \cite[\S 3.1]{kan1976every}; and another one uses the universal binate tower of groups \cite{berrick1989universal}, \cite{berrick1994binate}.  
\end{example}

\subsection{The category of abelian groups}

In this subsection we will study  subfunctors and quotients of the identity functor on the category of abelian groups. We will denote by $\ZZ(p^\infty)$ the $p$-quasicyclic group, which can be defined as $\ZZ(p^\infty) = \ZZ[1/p]/\ZZ,$ where $\ZZ[1/p]$ is the subgroup of $\QQ$ generated by elements $1/p^n.$

\begin{theorem} Let $F$ be a  non-vanishing, proper subfunctor of the identity functor on the category of abelian groups. Then  $\IM(F)$ contains either a non-trivial finite cyclic group or the $p$-quasicyclic group for some prime $p.$
\end{theorem}
\begin{proof} Assume the contrary, that  the  $\IM(F)$ contains no non-trivial finite cyclic groups or a $p$-quasicyclic group. 
Any abelian group $G$ can be embedded into a divisible abelian group $G\subseteq D.$ Any divisible abelian group $D$ is a direct sum of groups isomorphic to $\QQ$ and $\ZZ(p^\infty).$ Since a direct sum is embedded into the product, we obtain that any abelian group $G$ can be embedded into 
a product of copies of groups $\QQ$ and $\ZZ(p^\infty).$ In particular $\QQ/\ZZ\cong \bigoplus_{p} \ZZ(p^\infty)\subseteq \prod_{p} \ZZ(p^\infty).$ The group $\QQ$ can be embedded into the product $(\QQ/\ZZ)^{\mathbb{N}}$ by sending $q$ to $(q/1,q/2,q/3,\dots).$ Therefore, any abelian group $G$ can be embedded into a product of the $p$-quasicyclic groups 
\begin{equation}
 G \subseteq \prod_{\alpha \in A} \ZZ(p^\infty_\alpha).
\end{equation}
Since any proper subgroup of $\ZZ(p^\infty)$ is a finite cyclic group, the assumption implies $F(\ZZ(p^\infty))=0.$ Similarly by  Lemma \ref{lemma:subfunctor}, for any family of abelian groups $(G_\alpha)_{\alpha\in A}$ we have 
\begin{equation}
F( \prod_{\alpha\in A} G_\alpha ) \subseteq \prod_{\alpha\in A} F(G_\alpha).
\end{equation}
Since $F$ being, a sub-functor of the identity, it preserve injectivity of maps,  and hence it  follows that
\begin{equation}
 F(G) \subseteq   F(\prod_{\alpha \in A} \ZZ(p^\infty_\alpha)) \subseteq \prod_{\alpha \in A} F(\ZZ(p^\infty_\alpha))=0.   
\end{equation}
This  contradicts the non-vanishing assumption on $F.$ 
\end{proof}

\begin{corollary}\label{cor:ab1}
There is no a non-trivial subfunctor of the identity functor on the category of abelian groups taking values in the subcategory of 
torsion-free abelian groups.
\end{corollary}

\begin{theorem} Let $Q$ be a non-trivial quotient of the identity functor on the category of abelian groups. Then $\IM(Q)$ either contains a non-trivial finite cyclic group, or all free abelian groups.  
\end{theorem}
\begin{proof}
For a non-trivial quotient of the identity functor  $Q(G)=G/F(G)$ in the category of abelian groups $Q:{\sf Ab}\to {\sf Ab}$, we can construct a non-trivial quotient of the identity functor on the category of groups $\widetilde Q:{\sf Gr}\to {\sf Gr}$ by the formula $\widetilde Q(G)=Q(G/[G,G]).$ Then the assertion follows from Theorem \ref{th:gr:quotient}. 
\end{proof}

\begin{corollary}\label{cor:ab2}
There is no a non-trivial quotient of the identity functor on the category of abelian groups taking values in the subcategory of divisible groups.
\end{corollary}

\begin{example}[Subfunctor of the identity functor having values in a reflective subcategory] \label{example:reflective_abelian}
The analogue of Theorem \ref{th:reflective} does not hold for the category of abelian groups. Denote by $\RR_{2}$ the full subcategory of the category of abelian groups without $2$-torsion. It is easy to see that it is reflective with the reflector $A\mapsto A/{\sf tor}_2(A),$ where ${\sf tor}_p(A)=\{a\in A\mid \exists n : p^n a=0\}.$   Then ${\sf tor}_3(A)$ is a subfunctor of the identity functor having values in the reflective subcategory $\RR_{2}.$ (In the category of all groups this construction does not work  because a group generated by $3$-torsion can have $2$-torsion).
\end{example}

\begin{example}
 The category of abelian groups ${\sf Ab}$ differs from the category of groups in the following sense. Consider the variety of abelian groups $T_p$ consisting of groups $A$ such that $pA=0$. Then the identity functor $\Id_{\sf Ab}$ has both a non-trivial subfunctor and a non-trivial quotient taking values in $T_p$, defined as the kernel and the cokernel of the map of multiplication by $p$  
\begin{equation}
0\longrightarrow A[p] \longrightarrow A \overset{p\cdot}\longrightarrow A \longrightarrow A/pA \longrightarrow 0. 
\end{equation}   
Therefore, the analogue of Corollary \ref{noabeliansubgroups} does not work for abelian groups. 
\end{example}

\begin{example}[Nontrivial (co)augmented functors with values in $\QQ$-vector spaces]
 Corollary \ref{cor:ab1} and \ref{cor:ab2} imply that there are no non-trivial subfunctors and quotient-functors of the identity functor of the category of abelian groups taking values in the category of $\QQ$-vector spaces. However, there are both nontrivial augmented functors and 
  nontrivial coaugmented functors with values in $\QQ$-vector spaces:
  
 \begin{equation}
  A \longrightarrow A\otimes \QQ, \hspace{1cm}  {\sf Hom}(\QQ, A)\longrightarrow A.  
 \end{equation}
 The image of the map ${\sf Hom}(\QQ, A)\to A$ sending $f$ to $f(1),$ is the largest divisible subgroup $D(A)\subseteq A.$ Note that $\IM(D)$ does not contain non-trivial cyclic groups but it contains $\ZZ(p^\infty)$ for any $p.$
\end{example}


\begin{thebibliography}{99}

\bibitem{akhtiamov2021right} Akhtiamov, Danil, Sergei O. Ivanov, and Fedor Pavutnitskiy. "Right exact localizations of groups." {\it Israel Journal of Mathematics} 242.2 (2021): 839-873.

\bibitem{berrick1989universal} Berrick, A. J. "Universal groups, binate groups and acyclicity." Group Theory, Procs Singapore Conf, de Gruyter (Berlin, 1989). 1989.

\bibitem{berrick1994binate} Berrick, A. J., and K. Varadarajan. "Binate towers of groups." {\it Archiv der Mathematik} 62 (1994): 97-111.

\bibitem{casacuberta1992orthogonal}  Casacuberta, Carles, Georg Peschke, and Markus Pfenniger. "On orthogonal pairs in categories and localisation." {\it Adams Memorial Symposium on Algebraic Topology}. Vol. 1. 1992.

\bibitem{farjoun1996cellular} Farjoun, Emmanuel. {\it Cellular spaces, null spaces and homotopy localization.} No. 1622. Springer Science \& Business Media, 1996.

\bibitem{farjoun1996higher} Farjoun, E. Dror. "Higher homotopies of natural constructions." {\it Journal of Pure and Applied Algebra} 108.1 (1996): 23-34.

\bibitem{huq1969semivarieties} Huq, S. A. "Semivarieties and subfunctors of the identity functor." {\it Pacific Journal of Mathematics} 29.2 (1969): 303-309.

\bibitem{kan1976every} Kan, Daniel M., and William P. Thurston. "Every connected space has the homology of a $K (\pi, 1)$." {\it Topology} 15.3 (1976): 253-258.

\bibitem{kaplansky2020set} Kaplansky, Irving. {\it Set theory and metric spaces}. Vol. 298. American Mathematical Society, 2020.

\bibitem{libman2000cardinality} Libman, Assaf. "Cardinality and nilpotency of localizations of groups and G-modules." {\it Israel Journal of Mathematics} 117 (2000): 221-237.

\bibitem{schupp1976embeddings} Schupp, Paul E. "Embeddings into simple groups." {\it Journal of the London Mathematical Society} 2.1 (1976): 90-94.

\bibitem{shulman2008set} Shulman, Michael A. "Set theory for category theory." {\it arXiv preprint arXiv:0810.1279} (2008).

\bibitem{sierpinski1958cardinal} Sierpinski, Waclaw. "Cardinal and ordinal numbers." Monografie Matematyczne 34 (1958).


\end{thebibliography}
\end{document}